\documentclass{article}
\usepackage[utf8x]{inputenc}
\usepackage{amssymb,latexsym}
\usepackage{amsmath}
\usepackage{graphicx}
\usepackage{amsthm}
\usepackage[shortlabels]{enumitem}
\usepackage{appendix}
\usepackage{comment}
\usepackage{xcolor}

\newcommand{\ctau}{\tau^{(2)}}
\newcommand{\cnu}{\nu^{(2)}}

\newtheorem{theorem}{Theorem}[section]
\newtheorem{lemma}[theorem]{Lemma}

\newtheorem{conjecture}[theorem]{Conjecture}
\newtheorem{corollary}[theorem]{Corollary}

\newtheorem{claim}{Claim}[theorem]

\newtheorem{observation}[theorem]{Observation}

\newtheorem{notation}[theorem]{Notation}

\newtheorem{definition}[theorem]{Definition}

\newtheorem{question}[theorem]{Question}

\theoremstyle{remark}
\newtheorem{remark}[theorem]{Remark}

\newcommand{\refT}[1]{Theorem~\ref{#1}}
\newcommand{\refC}[1]{Corollary~\ref{#1}}
\newcommand{\refCl}[1]{Claim~\ref{#1}}
\newcommand{\refL}[1]{Lemma~\ref{#1}}

\newcommand{\refS}[1]{Section~\ref{#1}}

\newcommand{\refCon}[1]{Conjecture~\ref{#1}}

\newcommand{\refOb}[1]{Observation~\ref{#1}}

\makeatletter

\newcommand{\Rmnum}[1]{\expandafter\@slowromancap\romannumeral #1@}
\makeatother
\newcommand{\cI}{\Rmnum{1}}
\newcommand{\cII}{\Rmnum{2}}
\newcommand{\cO}{O}
\newcommand{\cT}{T}

\allowdisplaybreaks[4]

\usepackage{geometry}
\let\OLDthebibliography\thebibliography
\renewcommand\thebibliography[1]{
	\OLDthebibliography{#1}
	\setlength{\parskip}{0pt}
	\setlength{\itemsep}{0pt plus 0.3ex}
}

\title{$2$-covers of wide Young diagrams}
\author{Ron Aharoni\thanks{Faculty of Mathematics, Technion, Haifa 32000, Israel. E-mail: {\tt raharoni@gmail.com}}, Eli Berger\thanks{Department of Mathematics, University of Haifa, Haifa 31905,  Israel.  Email: {\tt berger.haifa@gmail.com}}, He Guo\thanks{Department of Mathematics and Mathematical Statistics, Ume\r{a} University, Ume\r{a} 90187, Sweden.  E-mail: {\tt he.guo@umu.se}. Research supported by the Kempe Foundation grant JCSMK23-0055.}, and Daniel Kotlar\thanks{Computer Science Department, Tel-Hai College, Upper Galilee 12210, Israel. Email: {\tt dannykot@telhai.ac.il}}}
\date{November 2023}

\begin{document}
	
	\maketitle
	\begin{abstract}
		A Young diagram $Y$ is called \emph{wide} if every sub-diagram $Z$ formed by a subset of the rows of $Y$ dominates $Z'$, the conjugate of $Z$. A Young diagram~$Y$ is called \emph{Latin} if there exists 
		an assignment of numbers to its squares, that is injective in each row and each column, and the $i$th row is assigned the numbers $1, \ldots ,a_i$, where $a_i$ is the length of the $i$th row. 
		A conjecture of Chow and Taylor, appearing in~\cite{chow}, is that a wide Young diagram is Latin.
		We prove a weaker dual version of the conjecture. 
	\end{abstract}

	\section{Introduction}
	Young diagrams are a way of representing partitions of natural numbers. Given a partition  $n=a_1+a_2+\cdots + a_m$ of a number $n$, where $a_1 \ge a_2 \ge \cdots \ge a_m\ge 0$, the corresponding Young diagram~$Y$ is formed by $m$ rows of squares that are left-aligned and the $i$th row consists of $a_i$ squares. The size $n$ of~$Y$ is denoted by~$|Y|$.
	The $i$th row (resp. column) of $Y$ is denoted by $r_i(Y)$ (resp. $c_i(Y)$).
	We write $a_i(Y)$ for $|r_i(Y)|$ and $b_i(Y)$ for $|c_i(Y)|$. 
	When the identity of $Y$ is clear from the context we omit its mention, and write $r_i, c_i, a_i, b_i$.

	The conjugate $Y'$ of a Young diagram $Y$ is obtained by transposing the diagram. In terms of partitions, this is the partition $n=b_1+b_2+\cdots+b_{a_1}$ with $b_1 \ge b_2\ge \cdots \ge b_{a_1}\ge 0$, where $b_j=|c_j|$, the length of the $j$th column of~$Y$.

	A Young diagram $X$ \emph{dominates} a diagram $Y$ partitioning the same $n$ if 
	\[\sum_{i=1}^k a_i(X)\ge \sum_{i=1}^ka_i(Y)\] 
	for every $k\ge 1$, where partitions are extended by appending zero parts at the end as necessary.
	
	\begin{definition}
		A Young diagram is said to be {\em wide} if for every subset $S$ of its rows, the diagram $Z$ formed by $S$  dominates $Z'$.      
	\end{definition}
	
	\begin{definition}
		A {\em filling} of a Young diagram $Y$ is an assignment of a number $y(i,j)$ to each square $(i,j)$ that is in the $i$th row and $j$th column of $Y$, satisfying that for each $i$, the assignment is an injection from the $a_i$ many squares in row $i$ to $[a_i]=\{1,\dots, a_i\}$.
		A filling is \emph{Latin} if it is injective also in each column.  
		A Young diagram is {\em Latin} if it has a Latin filling.
	\end{definition}

	In \cite{chow} the following is ascribed to Victor Reiner. (Later we prove a strengthening \refT{thm:converse}.) 
	\begin{theorem}\label{reiner}
		A Latin Young diagram is wide.
	\end{theorem}

	A conjecture attributed in \cite{chow} to Chow and Taylor is that the converse is also true:
	\begin{conjecture}\label{chowconj}
		If a Young diagram $Y$ is wide then it is Latin.
	\end{conjecture}
	
	In \cite{chow} this is called the Wide Partition Conjecture, or WPC. The original motivation for the conjecture comes from Rota's basis conjecture~\cite{rota}. Replacing the symbols with elements of a matroid and the condition of injectivity by matroid independence yields a generalization of Rota's conjecture. 
	
	The WPC has been open for over 20 years, with little progress toward a solution. A special case that has been solved is that of only two distinct part sizes~\cite[Theorem 3]{chow}. In this paper we study a re-formulation of the problem in terms of hypergraph matchings, and prove a weaker version in which matchings are replaced with covers.

	\subsection{A re-formulation in matching terminology}
	
	Following notation from \cite{az}, a \emph{$k$-matching} in a hypergraph~$H$ is a set of edges in which every two edges share fewer than $k$ vertices (so, a $1$-matching is a classical matching, i.e.,  a set of disjoint edges).
	The $k$th matching 
	number $\nu^{(k)}(H)$ of $H$ is the maximum size of a $k$-matching in~$H$. 
	Let $H^{(k)}$ be the hypergraph whose vertex set is $\{K\subseteq V(H): |K|=k, K\subseteq e\text{ for some $e\in E(H)$}\}$ and whose edge set is $\{ \binom{e}{k}: e \in E(H)\}$. Then a $k$-matching of $H$ is a matching in $H^{(k)}$, so $\nu^{(k)}(H)=\nu(H^{(k)})$, where $\nu(F)$ is the  matching number of a hypergraph~$F$.
	
	A \emph{$k$-cover} of~$H$ is a cover of $H^{(k)}$, namely a set $P$ of $k$-sets representing every edge of $H$, meaning that 
	for every edge~$e$ of~$H$, there exists some $K \in P$ such that $K\subseteq e$. 
	By $\tau^{(k)}(H)$ we denote the minimum size of a $k$-cover of $H$. Again, $\tau^{(k)}(H) = \tau(H^{(k)})$, where $\tau(F)$ is the  covering number of a hypergraph~$F$. 
	$\tau^{(k)}(H)$ is
	the integral LP-dual parameter of $\nu^{(k)}(H)$, hence  
	\begin{equation}\label{eq:tauknukrelation}
		\nu^{(k)}(H) \le \tau^{(k)}(H).
	\end{equation}

	Below we shall only consider the case $k=2$. We say that two sets $e_1,e_2$  are \emph{almost disjoint} if $|e_1\cap e_2|\le 1$.

	\begin{notation}
		\begin{enumerate}[(1)]
			\item Pairs and triples will be often written without delineating parentheses, so we write $ab$ for $\{a,b\}$ and   $abc$ for $\{a,b,c\}$.
			\item Throughout, a (hyper)graph is identified with its set of edges. 
		\end{enumerate}
	\end{notation}

	To a Young diagram $Y$ we assign a tripartite hypergraph $H(Y)$, as follows. Its three respective sides are $R=R(Y)=\{r_1, \ldots ,r_m\}$ (the set of rows), $C=C(Y)=\{c_1, \ldots, c_{a_1}\}$ (the set of columns), and $S=S(Y)=\{s_1, \ldots ,s_{a_1}\}$ (the set of numerical symbols). 
	
	For  $1\le i\le m$ and $r_i\in R$, let
	\begin{equation}\label{eq:def:Hi}
		H_i(Y)=\{r_ic_js_k \mid  c_j\in C, s_k\in S, 1\le j, k \le a_i\}.
	\end{equation}
	Let
	\[H(Y)=\bigcup_{1\le i \le m}H_i(Y).\]

	\begin{observation}\label{ob:tau2leY}
		$\tau^{(2)}(H(Y))\le |Y|$.
	\end{observation} This follows from the fact that $\{r_i c_j \mid 1\le i\le m, 1\le j\le a_i \}$ is a 2-cover for $H(Y)$.

	An equivalent formulation of the WPC is:
	\begin{conjecture}\label{conj:main}
		If a Young diagram $Y$ is wide, then $\cnu(H(Y))=|Y|$.  
	\end{conjecture}
	To see the equivalence, assume there is an assignment $y(i,j)$ satisfying the requirement of~\refCon{chowconj}.  Then 
	$M:=\{r_ic_js_{y(i,j)} \mid 1\le i\le m, 1\le j \le a_i\}$ is a $2$-matching in $H(Y)$ of size $|Y|$.
	To see the almost disjointness, note  that $r_ic_j$ is contained in one edge of~$M$ by the construction; since $y(i,j)$ is injective for fixed~$i$, then $r_is_k$ is contained in at most one edge of~$M$; since the assignment is injective in each column,  $c_js_k$ is contained in at most one edge of~$M$.
	For the other direction, let $M$ be a 2-matching  in $H(Y)$ of size $|Y|=\sum_{i=1}^m a_i$. Then for each $1\le i\le m$ and $1\le j\le a_i$, the pair $r_ic_j$ is contained in exactly one edge $r_ic_js_k$ of $M$. Then we set $y(i,j):=k$. Note that $1\le y(i,j)\le a_i$. Then $y$ is a Latin filling of $Y$, where the row injectivity and column injectivity requirements follow from the fact that $r_is_k$ and $c_js_k$ can each be contained in at most one edge of $M$, respectively. 
	
	In this terminology \refT{reiner} says that if $\nu^{(2)}(H(Y))=|Y|$ then $Y$ is wide. 
	
	Our main result is a weaker version of~\refCon{conj:main}:
	\begin{theorem}\label{thm:widetau2=n}
		If a Young diagram $Y$ is wide, then $\tau^{(2)}(H(Y))=|Y|$. 
	\end{theorem}

	The converse is also true: 
	\begin{theorem}\label{thm:converse}
		If $\tau^{(2)}(H(Y))=|Y|$, then the Young diagram $Y$ is wide.
	\end{theorem}
	This is a strengthening of \refT{reiner}: if $\nu^{(2)}(H(Y))=|Y|$ then by~\eqref{eq:tauknukrelation} and~\refOb{ob:tau2leY}, $\tau^{(2)}(H(Y))=|Y|$.
	By these theorems, $Y$ is wide if and only if $\tau^{(2)}(H(Y))=|Y|$.

	\refT{thm:converse} is proved in~\refS{sec:converse} and~\refT{thm:widetau2=n} in~\refS{sec:widetau2=n}. 
	In~\refS{sec:conjs} we post some questions strengthening~\refCon{conj:main}.
	
	
	\section{Preliminaries} \label{sec:prerequisites}

	\begin{notation}\label{nota:XtimesY}
		For two disjoint sets $A$ and $B$, we denote $\{ ab \mid a\in A, b\in B   \}$ by $A\times B$. 
	\end{notation}
	This notation is usually reserved for the set of ordered pairs, but the assumption $A \cap B =\emptyset$ removes the risk of ambiguity.

	Given  a tripartite hypergraph  $F$  with sides $A,B,C$, we denote $\{ bc\in B\times C \mid abc\in F \text{ for some $a\in A$}  \}$ by  
	$F[B\times C]$.

	\begin{notation}
		We abbreviate $\{c_1,\dots, c_i\}\times \{s_1,\dots, s_j\}\subseteq C\times S$ by $[i]\times [j]$.
	\end{notation}
	\begin{notation}
		Given an integer $\ell$ and a subset $Q$
		of $C\times S$, let \[\nu(\ell,Q):=\nu([\ell]\times[\ell]-Q)\]
		be the matching number of the graph $(\{c_1,\dots, c_\ell\}\times \{s_1,\dots, s_\ell\})\setminus Q$.
	\end{notation}

	\begin{lemma}\label{lemma:TinuaiQ}
		For any $Q\subseteq C\times S$ and $1\le i\le m$, the minimum number of pairs in $\{r_i\}\times (C\cup S)$ needed to cover all the edges of $H_i$ (defined in~\eqref{eq:def:Hi}) that are not covered by $Q$ is $\nu(a_i,Q)$.
	\end{lemma}
	\begin{proof}
		Let $T_i$ be a minimum collection of pairs in $\{r_i\}\times (C\cup S)$ that cover all the edges of $H_i$ that are not covered by $Q$. Then $\cup_{t\in T_i} t\cap(C\cup S)$ is a cover of the bipartite graph $[a_i]\times [a_i]-Q$, otherwise suppose $c_js_k$ is not covered by $\cup_{t\in T_i} t\cap(C\cup S)$, then $r_ic_js_k\in H_i$ is not covered by $T_i$ or $Q$, a contradiction. Hence 
		\[ |T_i|\ge \nu(a_i,Q),    \]
		by K\"onig's theorem.
		
		On the other hand, let $Z_i$ be a minimum cover of the bipartite graph $[a_i]\times [a_i]-Q$, it is routine to check that $\{r_i\}\times Z_i$ forms a 2-cover that covers all the edges of $H_i$ that are not covered by $Q$.
		Therefore \[|T_i|\le \nu(a_i,Q).  \]
		This proves that $\nu(a_i,Q)$ is the desired covering number.
	\end{proof}

	\begin{corollary}\label{cor:QtoP}
		For any $Q\subseteq C\times S$, there exists a 2-cover $P$ of $H(Y)$ satisfying 
		\[ \tau^{(2)}(H)\le  |P|=|Q|+\sum_{i=1}^m\nu(a_i,Q)   \]
		and $Q=P\cap (C\times S)$.
	\end{corollary}
	\begin{proof}
		Let $P=Q\cup\cup_{i=1}^m T_i$, where $T_i$ is as in the proof of~\refL{lemma:TinuaiQ}.
	\end{proof}

	\begin{observation}\label{ob:notcover}
		A pair in $\{r_i\}\times (C\cup S)$ does not cover any edge in $H_j$ for $i\neq j$.
	\end{observation}
	\begin{notation}
		For a set of pairs $P \subseteq (R\times C) \cup (R\times S) \cup (C \times S)$ let $Q(P)=P\cap(C\times S)$.
	\end{notation}
	
	\refL{lemma:TinuaiQ} and~\refOb{ob:notcover} imply the following.
	\begin{corollary}\label{cor:PtoQ}
		Let $P$ be a  2-cover  of $H(Y)$ and let $Q=Q(P)$. Then  \[|P|\ge |Q|+\sum_{i=1}^m\nu(a_i,Q).\]
	\end{corollary}

	\begin{lemma}\label{lemma:nupq}
		For $0\le p\le q$, we have 
		\begin{equation*}
			\nu(\ell, [p]\times [q])=\begin{cases}
				0, & \text{if $\ell\le p$}.\\
				\ell-p, & \text{if $p<\ell\le q$}.\\
				2\ell-p-q, & \text{if $q<\ell\le p+q$}.\\
				\ell & \text{if $\ell>p+q$}.
			\end{cases}
		\end{equation*}
	\end{lemma}
	\begin{proof}
		For $\ell\le p$, $[\ell]\times [\ell]- [p]\times [q]=\emptyset$ so that $\nu(\ell, [p]\times [q])=0$.
		
		For $p<\ell\le q$, in $[\ell]\times [\ell]- [p]\times [q]$ we can find a matching $\{c_{p+1}s_{p+1},\dots,c_\ell s_\ell \}$ of size $\ell-p$ and a vertex-cover $\{c_{p+1},\dots, c_\ell\}$ of the same size, implying $\nu(\ell, [p]\times [q])=\ell-p$.

		For $q<\ell\le p+q$, in $[\ell]\times [\ell]- [p]\times [q]$ we can find a matching 
		$\{c_{p+1}s_1,\dots,c_\ell s_{\ell-p}, s_{q+1}c_1,\dots, s_\ell c_{\ell-q} \}$
		of size $(\ell-p)+(\ell-q)$ and a vertex-cover $\{c_{p+1},\dots,c_\ell,s_{q+1},\dots, s_{\ell}\}$ of the same size, therefore $\nu(\ell, [p]\times [q])=2\ell-p-q$.
		
		In particular,  $\nu(p+q, [p]\times [q])=p+q$, which means that $[p+q]\times [p+q]-[p]\times[q]$ has a perfect matching. Then for $\ell>p+q$, this matching together with $\{c_{\ell'}s_{\ell'}\mid p+q<\ell'\le \ell   \}$ forms a  perfect matching in $[\ell]\times[\ell]-[p]\times[q]$. Therefore 
		$\nu(\ell, [p]\times [q])=\ell$ for $\ell>p+q$.
	\end{proof}

	\section{Proof of  \refT{thm:converse}}\label{sec:converse}
	
	We prove~\refT{thm:converse} in its contrapositive form. 
	\begin{theorem}\label{thm:inverse}
		If a Young diagram $Z$ is not wide, then $\tau^{(2)}(H(Z))<|Z|$. 
	\end{theorem}

	\begin{proof}[Proof of Theorem \ref{thm:inverse}]
		By the assumption that $Z$ is not wide, there exists a subpartition $Y$ formed by a set of rows of~$Z$, such that $Y$ does not dominate $Y'$. Let $k$ be the minimum number that witnesses this fact, namely 
		
		\begin{equation}\label{eq:not:balanced:k}
			\sum_{i =1}^k a_i(Y) < \sum_{i =1}^k b_i(Y).
		\end{equation}
		
		Write $a_i=a_i(Y)$ and $b_j=b_j(Y)$.

		\begin{claim} \label{akk}
			
			$k\le a_k.$
			
		\end{claim}
		\begin{proof}
			Assume the contrary and let $\ell = a_k<k$. We have
			\begin{equation}\label{eq:sumaibj}
				\sum_{i=1}^\ell  a_i= \sum_{i=1}^k  a_i- \sum_{\ell<i\le k}  a_i \quad \text{ and }\quad \sum_{j=1}^\ell  b_j= \sum_{j=1}^k  b_j- \sum_{\ell<j\le k}  b_j.  
			\end{equation}
			If we  show 
			\begin{equation}\label{eq:alarger}
				\sum_{\ell<i\le k}  a_i \ge   \sum_{\ell<i\le k}   b_j,
			\end{equation}
			then combining~\eqref{eq:sumaibj},~\eqref{eq:alarger}, and the assumption~\eqref{eq:not:balanced:k}, we have
			\[  \sum_{i=1}^\ell a_i < \sum_{j=1}^\ell b_j,  \]
			contradicting the minimality of~$k$.

			To prove \eqref{eq:alarger}, on the one hand,
			\begin{equation}\label{eq:aisum}
				\begin{split}
					\sum_{\ell<i\le k}  a_i
					\ge & \Bigl|\Big\{ (i,j)\in Y: \ell<i\le k, 1\le j\le \ell  \Big\}\Bigr|+ \underbrace{\Bigl|\Big\{ (i,j)\in Y: \ell<i\le k, \ell < j \le k  \Big\}\Bigr|}_{=:L_{\ell,k}} \\
					= & (k-\ell)\ell + L_{\ell,k},
				\end{split}
			\end{equation}
			where the equality is because $a_i\ge a_k=\ell$ for each $1\le i\le k$.
			
			On the other hand,
			\begin{equation}\label{eq:bjsum}
				\begin{split}
					\sum_{\ell<j\le k}  b_j 
					\le & \Bigl|\Big\{ (i,j)\in Y: 1\le i\le \ell, \ell<j\le k  \Big\}\Bigr|+\Bigl|\Big\{ (i,j)\in Y: \ell<i\le k, \ell < j \le k  \Big\}\Bigr|\\
					\le & (k-\ell)\ell + L_{\ell,k},
				\end{split}
			\end{equation}
			where the first inequality is because $a_k=\ell$, then for each $j>\ell$, $b_j=|\{i:a_i\ge j\}|\le k$ and then the squares in the $\ell+1,\dots,k$th columns of~$Y$ are contained in $\{(i,j)\in Y \mid 1\le i\le k, \ell<j\le k\}$.
			Combining~\eqref{eq:aisum} and~\eqref{eq:bjsum} yields~\eqref{eq:alarger}, and thus the claim. 
		\end{proof}

		\begin{claim}\label{y} 
			$ \ctau(H(Y))<\sum_{i=1}^ma_i=|Y|$. 
		\end{claim}
		\begin{proof}
			
			Let $Q=[k]\times [a_k]$, $T_0=\{i: a_i\le k, i\ge k+1\}$, $T_1=\{i: k<a_i\le a_k, i\ge k+1 \}$, $T_2=\{i: a_k\le a_i\le k+a_k, i\le k\}$, and $T_3=\{i: a_i> k+a_k\}$.

			By~\refC{cor:QtoP} and \refL{lemma:nupq} (that can be applied here in view of Claim \ref{lemma:nupq},)
			
			\begin{align}
				\ctau(H(Y))&\le |Q|+\sum_i\nu(a_i,Q)\nonumber\\
				&=ka_k+\sum_{i\in T_1}(a_i-k)+\sum_{i\in T_2}(2a_i-k-a_k)+\sum_{i\in T_3}a_i,\label{eq:Q+nu_i1}
			\end{align}
			and our aim is to show~\eqref{eq:Q+nu_i1} is less than $|Y|=\sum_ia_i$, 
			i.e.,
			\[  ka_k+\sum_{i\in T_1}(a_i-k)+\sum_{i\in T_2}(2a_i-k-a_k)+\sum_{i\in T_3}a_i<\sum_{i\in T_0}a_i+\sum_{i\in T_1}a_i+\sum_{i\in T_2}a_i+\sum_{i\in T_3}a_i.      \]
			This is equivalent to showing
			\begin{equation}\label{eq:Q+nu_i1equi}
				ka_k+\sum_{i\in T_2}(a_i-a_k)+\sum_{i\in T_3}(a_i-a_k)<\sum_{i\in T_0}a_i+k|T_1|+k|T_2|+\sum_{i\in T_3}(a_i-a_k).
			\end{equation}
			Noting that $|T_2\cup T_3|=k$, the left-hand side of~\eqref{eq:Q+nu_i1equi} is $\sum_{i\in T_2}a_i+\sum_{i\in T_3}a_i$, i.e., the sum of the lengths of the first $k$ rows of $Y$.
			Since $a_i-a_k\ge k$ for $i\in T_3$, the right-hand side is at least $\sum_{i\in T_0}a_i+k|T_1|+k|T_2|+k|T_3|$, which is the sum of the lengths of the first $k$ columns of $Y$. By our assumption~\eqref{eq:not:balanced:k},
			\[ \sum_{i\in T_2}a_i+\sum_{i\in T_3}a_i< \sum_{i\in T_0}a_i+k|T_1|+k|T_2|+k|T_3|.\]
			Therefore ,~\eqref {eq:Q+nu_i1equi} holds, completing the proof 
			of \refCl{y}. 
		\end{proof}    
		
		Now, let $P$ be the minimum $2$-cover of $H(Y)$. Then $|P|<|Y|$. Let $W = \{r_ic_j \mid (i,j) \in Z \setminus Y\}$. Then $P \cup W$ is a $2$-cover of $H(Z)$, of size $|P|+|Z|-|Y|$, which is less than $|Z|$. We prove the theorem.  
	\end{proof}

	\section{Wideness of $Y$ implies $\ctau(H(Y))=|Y|$}\label{sec:widetau2=n}
	
	In this section we prove our main result,  \refT{thm:widetau2=n}. Given a $2$-cover $P$ of $H(Y)$ for a wide diagram $Y$, we want to show that $|P|\ge |Y|$. As a first step, we replace $P$ by a $2$-cover having the same size, but more structure.  
	We ``shift" it, which in our context means that we ``push left'' the edges in $P\cap(C \times S)$ (the explicit definition is given below).
	\subsection{Shifting}
	Shifting is defined for an arbitrary set $P\subseteq (C\times S)\cup (R\times S)\cup (R\times C)$ with respect to either the $C$ or $S$ side of $H(Y)$ --- we choose the first, namely $C$. 
	Let $Q=Q(P)=P\cap(C\times S)$ and $W=W(P)=P\cap(C\times R)$.
	
	For a graph~$G$ and a vertex~$v$ of~$G$, let $N_G(v)$ be the set of neighbors of~$v$ in~$G$.
	
	For  an integer $t$
	let $L_t=\{r_\ell\in R: a_\ell <t \}$.

	Let $c_i,c_j$ be two elements of $C$ with $i<j$. The \emph{shifting} $P'=\phi_{c_i,c_j}(P)$ of $P$ is  obtained from $P$ by replacing all edges of $P$ incident with $\{c_i,c_j\}$ by 
	
	\begin{align*}
		&\{c_i\}\times \Big(N_{Q}(c_i)\cup N_{Q}(c_j) \Big)\bigcup \{c_j\}\times \Big(N_{Q}(c_j)\cap N_{Q}(c_i)\Big)\\
		\bigcup& \{c_i\}\times \Big(N_{W}(c_i)\cap \big(N_{W}(c_j)\cup L_j\big)\Big)\bigcup \{c_j\}\times \Big(N_{W}(c_j)\cup \big(N_{W}(c_i)\setminus L_j\big)\Big)
	\end{align*}

	\begin{lemma}\label{lemma:shiftingc}
		Let $i<j$ and let $P'=\phi_{c_i,c_j}(P)$. Then 
		\begin{enumerate}
			
			\item $|P'|=|P|$, and 
			\item If $P$ is a 2-cover of $H(Y)$, so is $P'$.
		\end{enumerate}

	\end{lemma}
	\begin{proof}
		The first part follows from 
		\[   |N_Q(c_i)\cup N_Q(c_j)|+ |N_Q(c_i)\cap N_Q(c_j)|=|N_Q(c_i)|+|N_Q(c_j)|,       \]
		and for $T_1=N_{W}(c_i)\cap(N_{W}(c_j)\cup L_j)$ and $T_2=N_{W}(c_j)\cup (N_{W}(c_i)\setminus L_j)$,
		\begin{align*}
			|T_1|+|T_2|=&|T_1\cup T_2|+|T_1\cap T_2|  \\
			=&|N_{W}(c_i)\cup N_{W}(c_j)|+|N_{W}(c_i)\cap N_{W}(c_j)|\\
			=&|N_{W}(c_i)| + |N_{W}(c_j)|.
		\end{align*}

		For the second part, since the pairs in~$P$ that are not incident to $c_i$ or $c_j$ are kept in~$P'$, we only need to show that edges of the form $e=r_\ell c_i s_k$ or $e=r_\ell c_j s_k$  
		are still covered by $P'$.

		{\bf Case I}:
		$e=r_\ell c_i s_k$.

		If $r_\ell s_k$ or $c_is_k$ is in $P$, then they are also in $P'$ and then~$e$ is covered by~$P'$.  So, we may assume that $r_\ell c_i \in P$. 
		If $r_\ell c_i \not \in  P'$, then by the definition of $P'$,  $r_\ell \not\in N_{W}(c_j)\cup L_j$.
		The fact that 
		$r_\ell \not\in L_j$ means that $j\le a_\ell$. Hence  $r_\ell c_js_k \in H(Y)$ ($k\le a_\ell$ follows from $e\in H(Y)$). 
		The fact that $r_\ell\not\in N_{W}(c_j)$ means that $r_\ell c_j s_k$ is covered by $P$ via $r_\ell s_k$ or $c_j s_k$. If $r_\ell s_k\in P$, then it is kept in $P'$; If $c_js_k \in P$, then by definition $c_i s_k \in P'$. In both cases, $e=r_\ell c_i s_k$ is covered by~$P'$ and we are done.
		
		{\bf Case II}:
		$e=r_\ell c_j s_k$.

		If $e$ is covered by $r_\ell s_k$ or by  $r_\ell c_j$ in $P$, then since these pairs are kept in $P'$, $e$ is covered in $P'$. Thus we may assume that $c_j s_k\in P$. On the other hand, the edge $r_\ell c_i  s_k$ of $H(Y)$ is covered by $P$.  If $r_\ell s_k \in P$, then $r_\ell s_k\in P'$ and $e$ is covered by~$P'$.  If $c_i s_k\in   P$, then $s_k\in N_Q(c_i)\cap N_Q(c_j)$, therefore by definition $c_j s_k \in P'$ and $e$ is covered by $P'$. If $r_\ell c_i \in P$, since $e\in H(Y)$, then $r_\ell\not\in L_j$ ($r_\ell c_j s_k\in H(Y)$ means that $a_\ell\ge j$). Therefore $r_\ell\in N_{W}(c_i)\setminus L_j$ so that by definition $r_\ell c_j \in P'$ and $e$ is covered by $P'$.
	\end{proof}

	By interchanging the roles of $C$ and $S$, we can also define the shifting operation $\phi_{s_i,s_j}(P)$ for $s_i,s_j\in S$ with~$i<j$, which satisfies  similar property as~\refL{lemma:shiftingc}.
	
	\begin{definition}\label{def:shifting-stable}
		A set $\Gamma\subseteq C\times S$ is called {\em closed down} (in $C \times S$)  if $c_ps_q \in \Gamma$ implies $[p]\times [q]\subseteq \Gamma$. 
	\end{definition}
	\begin{lemma}\label{lemma:Qshift}
		For every set $\Gamma\subseteq C \times S$ there exists a closed down $\Gamma'$  obtained from $\Gamma$ by a finite sequence of shifting operations. 
	\end{lemma}
	\begin{proof}
		We first show that after a finite sequence of shifting operations, we reach at a set that is stable under shifting. And then we prove that a stable set is close down.
		\begin{claim}
			After finitely many shifting operations on $\Gamma$, we can reach $\Gamma'\subseteq [a_1]\times [a_1]$ such that $\phi_{c_i,c_j}(\Gamma')=\phi_{s_i,s_j}(\Gamma')=\Gamma'$ for all $1\le i<j\le a_1$.
		\end{claim}
		\begin{proof}[Proof of the claim]
			Set $f(\Gamma')=\sum_{i=1}^{a_1}i\Big(|N_{\Gamma'}(c_i)|+|N_{\Gamma'}(s_i)|\Big)$. If $\phi_{c_i,c_j}$ or $\phi_{s_i,s_j}$ shifts $\Gamma'$ to $\Gamma''$ that is different from $\Gamma'$, then $f(\Gamma'')<f(\Gamma')$. But $f$ is non-negative, which means after finitely many shifting operations, the resulting set becomes stable, which completes the proof of the claim.
		\end{proof}
		
		\begin{claim}
			If $\phi_{c_i,c_j}(\Gamma')=\phi_{s_i,s_j}(\Gamma')=\Gamma'$ for all $1\le i<j\le a_1$, then $\Gamma'$ is closed down.
		\end{claim}
		\begin{proof}[Proof of the claim]
			Assume $\Gamma'$ satisfies the shifting stable property. If $c_ps_q \in \Gamma'$, then for all $p'\le p, ~q' \le q$, we have $c_{p'}s_{q'}\in \Gamma'$: If $p'<p$,  the shifting-stability applied to $c_p, c_{p'}$ yields $c_{p'}s_q \in \Gamma'$. If $q'< q$, then another application of the shifting-stability of $s_q,s_{q'}$ yields $c_{p'}s_{q'} \in \Gamma'$.
		\end{proof}
		Combining the claims, after finitely many shifting operations on $\Gamma$, we get a closed down~$\Gamma'$.
	\end{proof}

	Combining~\refL{lemma:shiftingc} and~\refL{lemma:Qshift}, we have the following:
	
	\begin{lemma}\label{lemma:shifted2cover}
		For any 2-cover $P$ of $H(Y)$, there exists a 2-cover $P'$ of $H(Y)$ such that $|P'|=|P|$ and $P'\cap (C\times S)$ is closed down.
	\end{lemma}

	\subsection{More facts about $|Q|$ and $\nu(\ell, Q)$}
	\refL{lemma:nupq} provides the values of $\nu(\ell,Q)$  when $Q=[p]\times [q]$. In this section, we study $\nu(\ell, Q)$ and $|Q|$ for general closed down $Q$. Together with~\refC{cor:PtoQ}, this will be the key in the proof of~\refT{thm:widetau2=n}.
	
	Throughout this section $Q\subseteq C \times S$
	and is assumed to be closed down. Let $p\ge 0$ be the maximum integer such that $[p]\times [p]\subseteq Q$. This is well-defined as $|Q|$ is finite.
	
	\begin{observation}\label{obs:jump_by_1_or_2}\hfill 
		\begin{enumerate}
			\item [(1)] $\nu(\ell,Q)=0$ for $\ell \le p$. 
			\item [(2)] $\nu(\ell,Q)-\nu(\ell-1,Q)\in\{1,2\}$ for every $\ell > p$.
		\end{enumerate}
		
	\end{observation}
	\begin{proof}
		The first part is obvious by the definition of~$p$. For the second part, for $\ell>p$, by maximality of~$p$ and the closed down property of~$Q$, we have $c_\ell s_\ell\not\in Q$.
		Then $\nu(\ell,Q)-\nu(\ell-1,Q) \ge 1$ is because for every matching in $[\ell-1] \times [\ell-1]$ avoiding $Q$, we can add the edge $c_{\ell}s_{\ell}$ to obtain a matching in   $[\ell] \times [\ell]$ avoiding $Q$. $\nu(\ell,Q)-\nu(\ell-1,Q) \le 2$ is because that adding two vertices to a bipartite graph cannot increase $\tau=\nu$ by more than~$2$. 
	\end{proof}

	Let $q\ge 0$ be minimal for which $\nu(\ell, Q)=\ell$ for all $\ell \ge p+q$. This is well-defined since it is easy to check for all $\ell\ge 2|Q|$, $\nu(\ell, Q)=\ell$. (See~\cite[Example 1.6 (4)]{looms} for a stronger bound.) For $p\ge 1$, for any $p\le \ell<p+p$, by~\refOb{obs:jump_by_1_or_2} we have $\nu(\ell, Q)\le 2(\ell-p)<\ell$. Hence $q\ge p$.

	By the second part of Observation~\ref{obs:jump_by_1_or_2}, we can divide the interval $(p,p+q]$ into (left-open and right-closed with integer endpoints) sub-intervals\footnote {Here   $O$ stands for ``One" and $T$  for ``Two". Similarly for the Roman letters $\cI$ and $\cII$ below.} $\cO_0,\cT_1,\cO_1,\cT_2,\dots,\cO_{k-1},\cT_k$ in order such that

	\begin{equation*}
		\nu(\ell,Q)-\nu(\ell-1,Q)=
		\begin{cases}
			1, & \text{if $\ell\in \cO_j$ for some $0\le j\le k-1$}.\\
			2, & \text{if $\ell\in \cT_j$ for some $1\le j\le k$}.
		\end{cases}
	\end{equation*}
	Note that all these sub-intervals, except possibly $\cO_0$, are non-empty. Furthermore, by the minimality of $q$, the last sub-interval must be $\cT_k$ (rather than some $\cO_j$).
	
	Let $\cI_j$ be the length of $\cO_j$ for $0\le j\le k-1$ and $\cII_j$ be the length of $\cT_j$ for $1\le j\le k$. 
	
	We can express $\nu(\ell, Q)$ explicitly in terms of  $\mathcal{Q}:=
	\{p,q,\cO_0,\cT_1,\dots,\cO_{k-1},\cT_k,\cI_0,\cII_1,\dots,\cI_{k-1},\cII_k\}$ of $Q$, as follows.

	\begin{lemma}
		\label{lemma:fcQexplicit}
		For any closed down $Q\subseteq C\times S$ with 
		$\mathcal{Q}$ defined as above, we have
		\begin{align}
			\cO_j&=(p+\sum_{\ell=0}^{j-1}\cI_\ell+\sum_{t=1}^j\cII_t, \; p+\sum_{\ell=0}^{j}\cI_\ell+\sum_{t=1}^j\cII_t],\label{eq:defOj}\\
			\cT_j&=(p+\sum_{\ell=0}^{j-1}\cI_\ell+\sum_{t=1}^{j-1}\cII_t,  \; p+\sum_{\ell=0}^{j-1}\cI_\ell+\sum_{t=1}^j\cII_t].\label{eq:defTj}
		\end{align}
		And   
		\begin{equation}\label{eq:explicitnucQ}
			\nu(\ell,Q)=\begin{cases}
				0, & \text{if $\ell\le p$}.\\
				\ell-p+\sum_{1\le t\le  j}\cII_t, & \text{if $\ell\in \cO_j$ for some $0\le j\le k-1$}.\\
				2\ell-2p-\sum_{0\le t<j}\cI_t, & \text{if $\ell\in \cT_j$ for some $1\le j\le k$}.\\
				\ell ,  &\text{if $\ell>p+q$}.
			\end{cases}
		\end{equation}
	\end{lemma}
	\begin{proof}
		The endpoints in \eqref{eq:defOj} and~\eqref{eq:defTj} follow from the definition of the intervals $O_j$ and $T_j$. The first and last items in \eqref{eq:explicitnucQ} follow from the definitions of $p$ and $q$.
		
		To compute $\nu(\ell,Q)$, when $\ell\in \cO_j$, let $\ell_0$ be the left endpoint of $\cO_j$, which is $p+\sum_{\ell=0}^{j-1}\cI_\ell+\sum_{t=1}^j\cII_t$ by~\eqref{eq:defOj}. And $\nu(\ell_0,Q)=\sum_{\ell=0}^{j-1}\cI_\ell+2\sum_{t=1}^j\cII_t$.
		Therefore
		\begin{align*}
			\nu(\ell,Q)&=(\ell-\ell_0)+\nu(\ell_0,Q)\\
			&=\Big( \ell- (p+\sum_{\ell=0}^{j-1}\cI_\ell+\sum_{t=1}^j\cII_t) \Big) + \sum_{\ell=0}^{j-1}\cI_\ell+2\sum_{t=1}^j\cII_t
			=   \ell-p+\sum_{1\le t\le  j}\cII_t.
		\end{align*}
		
		When $\ell\in \cT_j$, let $\ell_0=p+\sum_{\ell=0}^{j-1}\cI_\ell+\sum_{t=1}^{j-1}\cII_t$ be the left endpoint of $\cT_j$ by~\eqref{eq:defTj}. And $\nu(\ell_0,Q)=\sum_{\ell=0}^{j-1}\cI_\ell+2\sum_{t=1}^{j-1}\cII_t$.
		Therefore
		\begin{align*}
			\nu(\ell,Q)=2(\ell-\ell_0)+\nu(\ell_0,Q)= 2\ell-2p-\sum_{0\le t<j}\cI_t.
		\end{align*}
		Hence we complete the proof.
	\end{proof}

	We have the following results about the sums of $\cII_j$ and of $\cI_j$, and a lower bound on the size of~$Q$ with respect to $\cI_\ell,\cII_j$.
	\begin{corollary}\label{cor:IIkQ}
		For any closed down $Q\subseteq C\times S$ with $\mathcal{Q}$ defined as above,
		we have
		\begin{equation}\label{eq:sumII=P}
			\sum_{j=1}^k\cII_j=p,\quad \sum_{j=0}^{k-1}\cI_j=q-p,
		\end{equation}
	\end{corollary}

	\begin{proof}
		Note that by the definition of~$\cO_j,\cT_j,\cI_j,\cII_j$,
		\begin{align*}
			p+\sum_{j=0}^{k-1}\cI_j+\sum_{j=1}^k\cII_j=p+q=\nu(p+q,Q)=\sum_{j=0}^{k-1}\cI_j+2\sum_{j=1}^k\cII_j,
		\end{align*}
		from which the identities in \eqref{eq:sumII=P} immediately follow.
	\end{proof}
	
	\begin{lemma}\label{lem:Q_lower_bound}
		For any closed down $Q\subseteq C\times S$ with $\mathcal{Q}$ defined as above,
		\begin{equation}\label{eq:lbonsizeQ}
			|Q|\ge  \sum_{j=1}^k \cII_j\Big(p+\cI_0+\sum_{\ell=1}^{j-1}(\cII_\ell+\cI_\ell)\Big).  
		\end{equation}
	\end{lemma}
	
	\begin{proof}

		For $0\le j\le k-1$ and $\ell\in O_j$, we take a maximum matching $M_{\ell-1}$ in $([\ell-1]\times [\ell-1])\setminus Q$. The constraint $\nu(\ell,Q)-\nu(\ell-1,Q)=1$ implies that we cannot match both $c_\ell$ and $s_\ell$ to $\{c_1,\dots,c_{\ell-1}\}\cup\{s_1,\dots, s_{\ell-1}\}$ without using edges in~$Q$ or intersecting with edges in~$M_{\ell-1}$, otherwise we get a matching in $([\ell]\times[\ell])\setminus Q$ of size $|M_{\ell-1}|+2$. Therefore for one of $c_\ell$ and $s_\ell$, which we denote by $z_\ell$, say $z_\ell=s_\ell$, all the edges $\{c_1,\dots, c_{\ell-1}\}\times \{s_\ell\}$ are either in $Q$ or incident with edges in $M_{\ell-1}$. Since by~\eqref{eq:explicitnucQ}
		\begin{equation*}
			\begin{split}
				|M_{\ell-1}|=\nu(\ell-1,Q)=\nu(\ell,Q)-1 
				=\ell-1-p+\sum_{1\le t\le j}\cII_t,         
			\end{split}
		\end{equation*}
		then 
		\begin{equation}\label{eq:lbzloj}
			|N_{Q\cap([\ell-1]\times [\ell-1])}(z_\ell)|\ge (\ell-1)-|M_{\ell-1}|=p-\sum_{1\le t\le j}\cII_t.
		\end{equation}
		For any $1\le j\le k-1$, let $\ell_j$ be the largest integer in $\cT_j$.
		For all $\ell'\in \cT_j$, since $\ell'<\ell_j+1\in \cO_j$, by the closed down property of $Q$, 
		\[N_{Q\cap([\ell_j]\times[\ell_j])}(z_{\ell_j+1})\subseteq N_{Q\cap([\ell_j]\times[\ell_j])}(z_{\ell'}),\] 
		where $z_{\ell'}$ is set to be $s_{\ell'}$ if $z_{\ell_j+1} =s_{\ell_j+1}$, and to be $c_{\ell'}$ if $z_{\ell_j+1} =c_{\ell_j+1}$.
		Therefore
		\begin{equation}\label{eq:lbzltj}
			|N_{Q\cap([\ell_j]\times[\ell_j])}(z_{\ell'})|\ge   |N_{Q\cap([\ell_j]\times[\ell_j])}(z_{\ell_j+1})|\ge p-\sum_{1\le t\le j}\cII_t.
		\end{equation}
		\begin{remark}
			Since $Q$ is closed down and $p$ is maximal, the neighbors counted in~\eqref{eq:lbzloj} and~\eqref{eq:lbzltj} are the first vertices among the first $p$ vertices on the side opposite to the side of $z_\ell$ and $z_{\ell'}$. Counting the number of edges of~$Q$ from that side also leads to the lower bound on~$|Q|$ as in~\refC{cor:IIkQ} (for example, we can count $|Q|$ from the $C$-side of the $Q$ constructed in the proof of~\refC{cor:constructQ}) --- a fact not used in the proof.
		\end{remark}
		
		Hence by the choice of~$p$, combining~\eqref{eq:lbzloj} with~\eqref{eq:lbzltj}, and applying~\eqref{eq:sumII=P}, we have
		\begin{align*}
			|Q|
			&\ge \Bigl|[p]\times [p]\Bigr|+\sum_{\ell\in\cO_0}|N_{Q\cap([\ell-1]\times[\ell-1])}(z_\ell)|+\sum_{j=1}^{k-1}\Big(\sum_{\ell'\in\cT_j}|N_{Q\cap([\ell_j]\times[\ell_j])}(z_{\ell'})|+\sum_{\ell\in\cO_j}|N_{Q\cap([\ell-1]\times[\ell-1])}(z_\ell)|\Big)\\
			&\ge p(p+\cI_0)+\sum_{j=1}^{k-1} (\cII_j+\cI_j)(p-\sum_{1\le t\le j}\cII_t)\\
			&=\sum_{j=1}^k \cII_j(p+\cI_0)+\sum_{j=1}^{k-1} (\cII_j+\cI_j)\sum_{j<t\le k}\cII_t\\
			&=\sum_{j=1}^k \cII_j(p+\cI_0)+\sum_{t=2}^k\cII_t\Big( \sum_{j=1}^{t-1}(\cII_j+\cI_j) \Big)\\
			&=\sum_{j=1}^k \cII_j\Big(p+\cI_0+\sum_{\ell=1}^{j-1}(\cII_\ell+\cI_\ell)\Big),
		\end{align*}
		which completes the proof.
	\end{proof}
	
	The next result will not be used, but it is of independent interest.
	\begin{corollary}\label{cor:constructQ}
		For any integers $0\le p\le q$ and division of the interval $(p,p+q]$ into (left-open and right-closed with integer endpoints) sub-intervals $\cO_0,\cT_1,\cO_1,\cT_2,\dots,\cO_{k-1},\cT_k$ in order such that except possibly $\cO_0$, all the sub-intervals are non-empty, for $\cI_j=|\cO_j|$ and $\cII_j=|\cT_j|$, there exists a closed down $Q\subseteq C\times S$ (assuming $|C|=|S|\ge p+q$) satisfying~\eqref{eq:defOj}--\eqref{eq:sumII=P} and attaining the lower bound in~\eqref{eq:lbonsizeQ}, i.e.,
		\[  |Q|=\sum_{j=1}^k \cII_j\Big(p+\cI_0+\sum_{\ell=1}^{j-1}(\cII_\ell+\cI_\ell)\Big).\]
	\end{corollary}
	\begin{proof}
		We set 
		\[ Q:=\bigcup_{j=1}^k\Bigg( \{c_{p-\sum_{\ell=1}^{j}\cII_\ell+1},\dots,c_{p-\sum_{\ell=1}^{j-1}\cII_\ell}\}\times \Big(\{s_\ell:\ell\in [p]\cup \cO_0\}\cup \cup_{\ell=1}^{j-1}\{s_\ell:\ell\in \cT_\ell\cup \cO_\ell\} \Big) \Bigg).  \]
		It is routine to verify~$Q$ satisfies all the requirements.
	\end{proof}

	\subsection{Proof of~\refT{thm:widetau2=n}}
	
	We are now set to prove~\refT{thm:widetau2=n}. We prove the contrapositive statement.
	\begin{theorem}\label{thm:tau2notwide}
		If $\ctau(H(Y))<|Y|$, then the Young diagram $Y$ is not wide.
	\end{theorem}
	\subsubsection{A toy example}\label{subsubsec:toy}
	By~\refL{lemma:shifted2cover}, there exists a minimum 2-cover~$P$ of $H(Y)$ satisfying $P\cap(C\times S)$ is closed down and $|P|=\ctau(H(Y))<|Y|$. Furthermore, by~\refC{cor:PtoQ} and the assumption $|P|<|Y|$, $Q=P\cap(C\times S)$ is non-empty.
	In a toy case, we assume that
	$Q=[p]\times [q]$ for some $1\le p\le q$. For this $Q$, by~\refL{lemma:nupq} we have $\cO_0=(p,q]$, $\cT_1=(q, p+q]$, $\cI_0=q-p$, and $\cII_1=p$. Let $Z$ be the partition formed by the rows of $Y$ whose lengths are at most $p+q$. We shall show that the sum of the lengths of the first $\cII_1=p$ rows of $Z$ is less than that of the first $p$ columns of $Z$. Therefore $Z$ does not dominate $Z'$, hence  $Y$ is not wide.
	
	Indeed, let 
	\[D=\{i:a_i\le p \},\; A_{\cO_0}=\{i: a_i\in \cO_0\},\; A_{\cT_1}=\{i:a_i\in \cT_1 \}, \text{ and } U=\{i: a_i>p+q\}.  \]
	Let $A_{T_1}[\cII_1]$ be the set of the first $\cII_1=p$ elements of $A_{T_1}$ if $\cII_1\le |A_{T_1}|$, and to be $A_{T_1}$ if $\cII_1>|A_{T_1}|$.
	
	The first observation is
	\begin{equation}\label{eq:toyineq}
		\text{the sum of the lengths of the first $\cII_1$ rows of $Z$}\le pq+\sum_{i\in A_{\cT_1}[\cII_1]}(a_i-q).
	\end{equation}
	To see this, note that if $|A_{\cT_1}|\ge \cII_1=p$, then the RHS is equal to the LHS, which is $\cII_1\cdot q+\sum_{i\in A_{\cT_1}[\cII_1]}(a_i-q)=\sum_{i\in A_{\cT_1}[\cII_1]}a_i$. If  $|A_{\cT_1}|< \cII_1$, then $|A_{\cT_1}|q+\sum_{i\in A_{\cT_1}[\cII_1]}(a_i-q)$  is the sum of the lengths of the first $|A_{\cT_1}|$ rows of~$Z$, and the sum of the lengths of the remaining $\cII_1-|A_{\cT_1}|$ rows is at most $(p-|A_{\cT_1}|)q$, because each of these remaining rows is below those with indices in~$A_{\cT_1}$ and has length at most $q$.
	
	By~\refC{cor:PtoQ} and~\refL{lemma:nupq}, 
	we have
	\begin{align*}
		\ctau(H(Y))=|P|&\ge |Q|+\sum_i\nu(a_i,Q)\\
		&=pq+\sum_{i\in A_{\cO_0}}(a_i-p)+\sum_{i\in A_{\cT_1}[\cII_1]}(2a_i-p-q)+\sum_{i\in A_{\cT_1}\setminus A_{\cT_1}[\cII_1]}(2a_i-p-q)+\sum_{i\in U}a_i.
	\end{align*}
	Together with the assumption $\ctau(H(Y))<|Y|=\sum_{i\in D\cup A_{\cO_0}\cup A_{\cT_1}\cup U}a_i$, we have
	\begin{equation}\label{eq:toyeq1}
		pq+\sum_{i\in A_{\cO_0}}(a_i-p)+\sum_{i\in A_{\cT_1}[\cII_1]}(2a_i-p-q)+\sum_{i\in A_{\cT_1}\setminus A_{\cT_1}[\cII_1]}(2a_i-p-q)+\sum_{i\in U}a_i< \sum_{i\in D\cup A_{\cO_0}\cup A_{\cT_1}\cup U}a_i.    
	\end{equation}
	Note that
	\begin{align}
		&\text{LHS of~\eqref{eq:toyeq1}}\nonumber\\
		=& \sum_{i\in A_{\cO_0}}a_i + \sum_{i\in A_{T_1}}a_i +\Big(pq+ \sum_{i\in A_{\cT_1}[\cII_1]}(a_i-q)\Big)+\sum_{i\in  A_{T_1}\setminus A_{T_1}[\cII_1]}(a_i-q)+\sum_{i\in U}a_i-p(|A_{\cO_0}|+|A_{\cT_1}|)\nonumber\\
		\ge& \text{ the sum of the lengths of the first $\cII_1$ rows of $Z$ } + \sum_{i\in A_{\cO_0}\cup A_{\cT_1}\cup U} a_i - p(|A_{\cO_0}|+|A_{\cT_1}|),\label{eq:toyeq2}
	\end{align}
	where the last inequality is by~\eqref{eq:toyineq} and the fact that for each $i\in A_{\cT_1}\setminus A_{\cT_1}[\cII_1]$, $a_i\ge q$.
	Then comparing~\eqref{eq:toyeq2} with the RHS of~\eqref{eq:toyeq1}, after cancellation, we have
	\[\text{the sum of the lengths of the first $\cII_1$ rows of $Z$}< \sum_{i\in D}a_i+p(|A_{\cO_0}|+|A_{\cT_1}|), \]
	where since the rows with indices in $A_{\cO_0}\cup A_{\cT_1}$ have lengths greater than~$p$, the RHS is the sum of the lengths of the first $\cII_1=p$ columns of $Z$. We are done.
	
	\subsubsection{Proof of~\refT{thm:tau2notwide}}\label{subsec:mainproof}
	\begin{proof}[Proof of~\refT{thm:tau2notwide}]
		By~\refL{lemma:shifted2cover}, there exists a minimum 2-cover~$P$ of $H(Y)$ satisfying $P\cap(C\times S)$ is closed down and $|P|=\ctau(H(Y))<|Y|$. Let $Q=P\cap(C\times S)$. By~\refC{cor:PtoQ} and the assumption $|P|<|Y|$, we have $Q\neq\emptyset$.
		By~\refC{cor:PtoQ} and~\refL{lem:Q_lower_bound}, and using notations there, where the non-emptiness of~$Q$ implies $1\le p\le q$, we have
		\begin{equation}\label{eq:fQ<n}
			\begin{split}
				f(Q)&:= \sum_{j=1}^k \cII_j\Big(p+\cI_0+\sum_{\ell=1}^{j-1}(\cII_\ell+\cI_\ell)\Big)+\sum_{i}\nu(a_i,Q)\\
				&\le |Q|+\sum_i\nu(a_i,Q)\le |P|<\sum_i a_i,
			\end{split}
		\end{equation}
		where, as above, $a_i$ is the length of the $i$th row of $Y$.

		Using the notation in~\refL{lemma:fcQexplicit}, let 
		\begin{align*}
			&D=\{i:a_i\le p \},\\ 
			&A_{\cO_j}=\{i: a_i\in \cO_j\} \text{ for $0\le j\le k-1$,}\\
			&A_{\cT_j}=\{i:a_i\in \cT_j \}\text{ for $1\le j\le k$},\\
			&U=\{i: a_i>p+q\}.
		\end{align*}
		For $1\le j\le k$, let $Z_j$ be the subpartition of $Y$ that consists of the rows with index in $D\cup_{t=0}^{j-1}A_{O_t}\cup_{\ell=1}^jA_{T_\ell}$. We define $A_{T_j}[\cII_j]$ to be the set of the first $\cII_j$ elements of $A_{T_j}$ if $\cII_j\le |A_{T_j}|$, and to be $A_{T_j}$ if $\cII_j>|A_{T_j}|$.
		
		We aim to show that the sum of the lengths of the first $\cII_k$ rows of $Z_k$ is less than that of the first $\cII_k$ columns of $Z_k$, which implies that~$Z_k$ does not dominate~$Z_k'$ and then~$Y$ is not wide.
		
		(Note that in the toy example in~\refS{subsubsec:toy}, $k=1$ and we set $Z=Z_k=Z_1$.)

		As before, the first observation is an upper bound on the sum of the lengths of the first $\cII_j$ rows of $Z_j$.
		\begin{claim}\label{claim:IIjHj}
			For every $1\le j\le k$,   the sum of the lengths of the first $\cII_j$ rows of $Z_j$ is at most
			\begin{equation*}
				h_j:=   \cII_j\Big(p+\cI_0+\sum_{\ell=1}^{j-1}(\cII_\ell+\cI_\ell)\Big)+\sum_{i\in A_{\cT_j}[\cII_j]}(a_i-2p-\sum_{0\le t<j}\cI_t+\sum_{\ell=j}^k\cII_\ell).
			\end{equation*}
		\end{claim}
		Note that the first item is part of $f(Q)$ in~\eqref{eq:fQ<n}.
		\begin{proof}[Proof of~\refCl{claim:IIjHj}]
			To verify this, first observe that since $p-\sum_{\ell=j}^k\cII_\ell=\sum_{\ell=1}^{j-1}\cII_\ell$ by~\eqref{eq:sumII=P}, then the term in the second summand of $h_j$ satisfies
			\begin{align*}
				&a_i-2p-\sum_{0\le t<j}\cI_t+\sum_{\ell=j}^k\cII_\ell\\
				=& a_i-p-\sum_{0\le t<j}\cI_t-(p-\sum_{\ell=j}^k\cII_\ell)\\
				=&a_i-p-\sum_{0\le t<j}\cI_t-\sum_{\ell=1}^{j-1}\cII_\ell =a_i-\Big(p+\cI_0+\sum_{\ell=1}^{j-1}(\cII_\ell+\cI_\ell)\Big).
			\end{align*}
			
			Note that $w_j:=p+\cI_0+\sum_{\ell=1}^{j-1}(\cII_\ell+\cI_\ell)\textit{}$ appears in the first part of the definition of~$h_j$.
			
			If $\cII_j\le |A_{T_j}|$, then $h_j=\cII_jw_j+\sum_{i\in A_{\cT_j}[\cII_j]}(a_i-w_j)=\sum_{i\in A_{\cT_j}[\cII_j]}a_i$, which is the sum of the lengths of the first $\cII_j$ rows of $Z_j$. 
			
			If $\cII_j> |A_{T_j}|$, then
			\begin{align*}
				h_j=&   |A_{\cT_j}|\cdot w_j+\sum_{i\in A_{\cT_j}}(a_i-w_j)+(\cII_j-|A_{\cT_j}|)\cdot w_j\\
				=&\Big(\sum_{i\in A_{\cT_j}}a_i\Big)+(\cII_j-|A_{\cT_j}|)\cdot w_j.
			\end{align*}
			Since the $|A_{\cT_j}|+1,\dots,\cII_j$th rows of $Z_j$ are below the rows with indices in $A_{\cT_j}$,  by~\eqref{eq:defTj} in~\refL{lemma:fcQexplicit}, each of these $\cII_j-|A_{\cT_j}|$ rows has size at most $w_j=p+\cI_0+\sum_{\ell=1}^{j-1}(\cII_\ell+\cI_\ell)$. Therefore in this case, the sum of the lengths of the first $\cII_j$ rows of $Z_j$ is at most $h_j$.
		\end{proof}

		Together with~\eqref{eq:explicitnucQ} in~\refL{lemma:fcQexplicit}, we can rewrite the representation of $f(Q)$ in~\eqref{eq:fQ<n} as  follows.
		\begin{align}
			f(Q)=&\sum_{j=1}^k \cII_j\Big(p+\cI_0+\sum_{\ell=1}^{j-1}(\cII_\ell+\cI_\ell)\Big)\nonumber\\
			&+\sum_{j=0}^{k-1}\sum_{i\in A_{\cO_j}}(a_i-p+\sum_{1\le t\le  j}\cII_t)+
			\sum_{j=1}^k\sum_{i\in A_{\cT_j}}(2a_i-2p-\sum_{0\le t<j}\cI_t)+\sum_{i\in U}a_i\nonumber\\
			=&\sum_{j=1}^k \cII_j\Big(p+\cI_0+\sum_{\ell=1}^{j-1}(\cII_\ell+\cI_\ell)\Big)\nonumber\\
			&+\sum_{j=0}^{k-1}\sum_{i\in A_{\cO_j}}(a_i-p+\sum_{1\le t\le  j}\cII_t)+\sum_{j=1}^k\sum_{i\in A_{\cT_j}}a_i\nonumber\\
			&+\sum_{j=1}^k\Big(\sum_{i\in A_{\cT_j}[\cII_t]}+\sum_{i\in A_{\cT_j}\setminus A_{\cT_j}[\cII_j]}\Big)(a_i-2p-\sum_{0\le t<j}\cI_t+\sum_{\ell=j}^k\cII_\ell)\nonumber\\
			&-\sum_{j=1}^k|A_{\cT_j}|\sum_{\ell=j}^k\cII_\ell+\sum_{i\in U}a_i\nonumber\\
			=&\sum_{j=1}^kh_j+\sum_{j=0}^{k-1}\sum_{i\in A_{\cO_j}}\Big(a_i-(p-\sum_{1\le t\le  j}\cII_t)\Big)+\sum_{j=1}^k\sum_{i\in A_{\cT_j}}a_i\nonumber\\
			&+\sum_{j=1}^k\sum_{i\in A_{\cT_j}\setminus A_{\cT_j}[\cII_j]}\Big(a_i-p-\sum_{0\le t<j}\cI_t-(p-\sum_{\ell=j}^{k}\cII_\ell)\Big)\nonumber\\
			&-\sum_{j=1}^k|A_{\cT_j}|\sum_{\ell=j}^k\cII_\ell +\sum_{i\in U}a_i\nonumber\\
			=&\sum_{j=1}^kh_j+\sum_{i\in Z_k\cup U\setminus D}a_i-\sum_{j=1}^k|A_{\cT_j}|\sum_{\ell=j}^k\cII_\ell-\sum_{j=0}^{k-1}|A_{\cO_j}|\sum_{\ell=j+1}^k\cII_\ell\nonumber\\
			&+\sum_{j=1}^k\sum_{i\in A_{\cT_j}\setminus A_{\cT_j}[\cII_j]}(a_i-p-\sum_{0\le t<j}\cI_t-\sum_{\ell=1}^{j-1}\cII_\ell).\nonumber
		\end{align}
		For the last term, for $i\in A_{\cT_j}$, by~\eqref{eq:defTj} we have $a_i-p-\sum_{0\le t<j}\cI_t-\sum_{\ell=1}^{j-1}\cII_j\ge 0$. Therefore
		\begin{equation}\label{eq:fQlowerboundside}
			f(Q)\ge  \sum_{j=1}^kh_j+\sum_{i\in Y_k\cup U\setminus D}a_i-\sum_{j=1}^k|A_{\cT_j}|\sum_{\ell=j}^k\cII_\ell-\sum_{j=0}^{k-1}|A_{\cO_j}|\sum_{\ell=j+1}^k\cII_\ell. 
		\end{equation}  
		
		On the other hand, by~\eqref{eq:fQ<n} we have
		\begin{equation}\label{eq:fQsumai}
			f(Q)<  \sum_{i\in Z_k\cup U}a_i.
		\end{equation}
		Combining~\eqref{eq:fQlowerboundside} with~\eqref{eq:fQsumai}, we have
		\begin{equation}\label{eq:sumHj}
			\begin{split}
				\sum_{j=1}^kh_j &< \sum_{i\in D}a_i + \sum_{j=1}^k|A_{\cT_j}|\sum_{\ell=j}^k\cII_\ell+\sum_{j=0}^{k-1}|A_{\cO_j}|\sum_{\ell=j+1}^k\cII_\ell\\
				&=\sum_{i\in D} a_i +\sum_{j=1}^k\cII_j(\sum_{1\le t\le j}|A_{\cT_t}|+\sum_{0\le \ell<j}|A_{\cO_\ell}| ).
			\end{split}
		\end{equation}
		We may assume that for each $1\le j\le k-1$, the sum of the lengths of the first $\cII_j$ rows of $Z_j$ is at least that of the first $\cII_j$ columns of $Z_j$, which is 
		\[\cII_j(\sum_{1\le t\le j}|A_{\cT_t}|+\sum_{0\le \ell <j}|A_{\cO_\ell}|)+\sum_{i:\: a_i\le \cII_j}a_i+\sum_{i:\: a_i>\cII_j, a_i\in D}\cII_j,\]
		otherwise $Z_j$ does not dominate $Z_j'$ and we already prove $Y$ is not wide.
		(Note that the first two summands arise from the fact that the lengths of the corresponding rows are greater than $p$ and $p\ge II_j$ by \eqref{eq:sumII=P}.)
		Hence by~\refCl{claim:IIjHj}, we have for every $1\le j\le k-1$,
		\begin{equation}\label{eq:eachHj}
			h_j\ge \cII_j(\sum_{1\le t\le j}|A_{\cT_t}|+\sum_{0\le \ell <j}|A_{\cO_\ell}|)+\sum_{i:\: a_i\le \cII_j}a_i+\sum_{i:\:a_i>\cII_j, a_i\in D}\cII_j.
		\end{equation}
		Combining~\eqref{eq:sumHj} with~\eqref{eq:eachHj} and subtracting $\sum_{j=1}^{k-1}h_j$ from $\sum_{j=1}^k h_j$, we have
		\begin{equation}
			h_k< \cII_k(\sum_{t=1}^k|A_{\cT_t}|+\sum_{\ell=0}^{k-1}|A_{\cO_\ell}|)+ \sum_{i\in D}a_i - \sum_{j=1}^{k-1}(\sum_{i:\: a_i\le \cII_j}a_i+\sum_{i:\: a_i>\cII_j, a_i\in D}\cII_j).
		\end{equation}
		If we can show 
		\begin{equation}\label{eq:aiDrelation}
			\sum_{i\in D}a_i - \sum_{j=1}^{k-1}(\sum_{i:\: a_i\le \cII_j}a_i+\sum_{i:\: a_i>\cII_j, a_i\in D}\cII_j)\le \sum_{i:\: a_i\le \cII_k}a_i+\sum_{i:\: a_i>\cII_k, a_i\in D}\cII_k, 
		\end{equation} 
		then by~\refCl{claim:IIjHj}, 
		we have the sum of the lengths of the first $\cII_k$ rows of $Z_k$ is at most
		\[ h_k< \cII_k(\sum_{t=1}^k|A_{\cT_t}|+\sum_{\ell=0}^{k-1}|A_{\cO_\ell}|)+  \sum_{i:\: a_i\le \cII_k}a_i+\sum_{i:\: a_i>\cII_k, a_i\in D}\cII_k,\]
		where the right-hand side is the sum of the lengths of the first $\cII_k$ columns of $Z_k$ and we complete the proof.
		
		To see that~\eqref{eq:aiDrelation} is true, we consider an equivalent form
		\begin{equation}\label{aiDrelationequiv}
			\sum_{i\in D}a_i \le  \sum_{j=1}^{k}\sum_{i:\: a_i\le \cII_j}a_i+\sum_{j=1}^{k}\cII_j\sum_{i:\: a_i>\cII_j, a_i\in D}1.
		\end{equation}
		Indeed, for each $a_i$ counted by the left-hand side, if $a_i\le \max(\cII_1,\dots,\cII_k)$, then it appears at least once in the first summand on the right-hand side. If $p\ge a_i > \max(\cII_1,\dots,\cII_k)$, then it is counted $k$ times in the second summand of the right-hand side, which by~\eqref{eq:sumII=P} contribute $\sum_{j=1}^k\cII_j=p\ge a_i$ to the right-hand side. Therefore we prove~\eqref{aiDrelationequiv} and then~\eqref{eq:aiDrelation}, which completes the proof of the theorem.
	\end{proof}

	\section{$\ctau$ vs. $\cnu$ and some stronger questions}\label{sec:conjs}
	\begin{observation}
		$\tau(H(Y))=\nu(H(Y))$.
	\end{observation} 
	\begin{proof}
		Let $G=H[R \times C]$ (as defined in~\refS{sec:prerequisites}). 
		We claim that $ \tau(H(Y))= \nu(H(Y))=\tau(G)$. 
		Clearly,   $ \tau(H(Y))\le \tau(G)$, and since $\nu(H(Y))\le \tau(H(Y))$, by K\"onig's theorem it suffices to show that $\nu(H(Y))=\nu(G)$. 
		As noted above, $ \nu(H(Y))\le \nu(G)$. 
		To show the converse, for any matching $M$  in $G$,  let $\tilde{M}=\{rc_js_j \mid rc_j \in M\}$, where $C=\{c_1, \ldots ,c_{a_1}\}$ and $S=\{s_1, \ldots ,s_{a_1}\}$. Then, by the definition of $H(Y)$, $\tilde{M}$ is  a matching in $H(Y)$, and $|\tilde{M}|=|M|$.
	\end{proof}
	
	If the wide diagram conjecture is true, then also $\ctau(H(Y))=\cnu(H(Y))$ whenever $Y$ is wide. 
	We do not know an example, wide or not, in which this integral duality relation fails. 
	
	\begin{question}\label{question:tau2=nu2}
		Is it true that $\ctau(H(Y))=\cnu(H(Y))$ for every Young diagram $Y$?
	\end{question}

	Based on~\refT{thm:widetau2=n}, solving \refCon{chowconj} is equivalent to solving this question, restricted to wide diagrams.

	\small
	\bibliographystyle{abbrv}

	\normalsize

	
	

\end{document}